\documentclass[11pt]{article}
\usepackage{amsmath,amsthm,amsfonts,amssymb,amscd, amsxtra, mathrsfs}
\usepackage{url}
\usepackage{cancel} 
\usepackage{rotating}
\usepackage[margin=2.7 cm,nohead]{geometry}
\usepackage{color}
\usepackage{color}
\usepackage[pdftex]{hyperref}
\usepackage{graphicx,color,rotating}
\usepackage{amsmath,amsthm}
\usepackage{amssymb}

\theoremstyle{plain}
\newtheorem{theorem}{Theorem}[section]
\newtheorem{lemma}[theorem]{Lemma}

\newtheorem{proposition}[theorem]{Proposition}

\theoremstyle{definition}
\newtheorem{definition}[theorem]{Definition}
\newtheorem{example}[theorem]{Example}

\theoremstyle{remark}
\newtheorem{remark}{Remark}

\newtheorem{algorithm}{Algorithm}

\newtheorem*{armijo}{\bf Armijo stepsize}

\newtheorem*{adaptative}{\bf Lipschitz-based adaptive stepsize}
\newtheorem*{diminishing}{\bf Diminishing stepsize}

\def\diam{\operatorname{diam}}

\newcommand{\argmin}{\rm argmin}

\def \T {{\scriptscriptstyle\mathrm{T}}} 

\def \T {{\scriptscriptstyle\mathrm{T}}} 

\begin{document}
\title{Frank--Wolfe algorithm for star-convex functions }
\author{
R. D\'iaz Mill\'an \thanks{School of Information Technology, Deakin University, Melbourne, Australia, e-mail:\url{r.diazmillan@deakin.edu.au}.}
\and
O.  P. Ferreira  \thanks{IME, Universidade Federal de Goi\'as,  Goi\^ania, Brazil, e-mail:\url{orizon@ufg.br}.}
\and
 J. Ugon  \thanks{School of Information Technology, Deakin University, Melbourne,   Australia,  e-mail:\url{j.ugon@deakin.edu.au}.}}
\maketitle
\begin{abstract}
We study the Frank--Wolfe algorithm for minimizing a differentiable function with Lipschitz continuous gradient over a compact convex set. To extend classical complexity bounds to certain non-convex functions, we focus on the class of \emph{star-convex functions}, which retain essential geometric properties despite the lack of convexity. We establish iteration-complexity bounds of $\mathcal{O}(1/k)$ for both the objective values and the duality gap under star-convexity, using diminishing, Armijo-type, and Lipschitz-based stepsize rules. Notably, the diminishing and Armijo strategies do not require prior knowledge of Lipschitz or curvature constants. These results demonstrate that the Frank--Wolfe method preserves optimal complexity guarantees beyond the convex setting.
\end{abstract}


\section{Introduction}

The Frank--Wolfe algorithm, also known as the conditional gradient method, has a long and influential history, beginning with its introduction in the 1950s to solve constrained convex quadratic programs over polyhedral sets~\cite{FrankWolfe1956}. A decade later, it was extended to minimize general convex functions with Lipschitz continuous gradients over compact convex domains~\cite{LevitinPolyak1966}. The method gained renewed interest in recent years due to its simplicity, low memory footprint, and projection-free structure, making it especially suitable for large-scale and high-dimensional problems. Its effectiveness in exploiting problem structure, such as separability and sparsity, has led to a proliferation of variants and theoretical advances (see, for example,~\cite{BeckTeboulle2004,Bouhamidietall2018,BoydRecht2017,FreundMazumder2017, Ghadimi2019,HarchaouiNemirovski2015, Jaggi2013, Konnov2018, LanZhou2016,LussTeboulle2013}).

In this work, we investigate the application of the Frank--Wolfe method to a broad class of non-convex optimization problems of the form
\[
\min_{x \in \mathcal{C}} f(x),
\]
where $\mathcal{C} \subset \mathbb{R}^n$ is a compact convex set and $f:\mathbb{R}^n \to \mathbb{R}$ is a differentiable function with Lipschitz continuous gradient. We are particularly interested in cases where the objective function $f$ is not convex but belongs to a class of functions that extend the notion of convexity while still allowing for global convergence guarantees. Specifically, we focus on the class of \emph{star-convex functions}, a notion introduced in~\cite{NesterovPolyak2006} in the context of second-order methods, which turns out to be especially well-suited for analyzing the Frank--Wolfe algorithm in non-convex settings.

Despite the non-convexity of star-convex functions, we show that the Frank--Wolfe method achieves a convergence rate of $\mathcal{O}(1/k)$ for both the function values and the duality gap, provided that $f$ satisfies star-convexity and has a Lipschitz continuous gradient. This extends prior work which typically guarantees only a rate of $\mathcal{O}(1/\sqrt{k})$ for general non-convex functions~\cite{Lacoste2016}. We propose a version of the Frank--Wolfe algorithm equipped with an adaptive stepsize rule that does not require any estimate of the Lipschitz constant. This rule is inspired by techniques developed in~\cite{Beck2015} (see also~\cite{BeckTeboulle2009, PedregosaJaggi2020}). Unlike previous analyses that rely on curvature or global smoothness bounds, our method adaptively estimates local descent parameters using only function and gradient evaluations. As a result, it remains efficient and broadly applicable, especially in large-scale settings.

The remainder of this paper is organized as follows. Section~\ref{sec:Preliminares} introduces the necessary background and notation. In Section~\ref{sec:StarConvex}, we define the class of star-convexity functions and discuss their key properties. Section~\ref{sec;OptProbl} presents the optimization problem under consideration, along with the assumptions and relevant properties. In Section~\ref{Sec:FW}, we describe the Frank--Wolfe algorithm and establish iteration-complexity bounds under star-convexity. Finally, Section~\ref{sec:conclusions} offers concluding remarks and directions for future research.
\section{Preliminaries} \label{sec:Preliminares}
In this section, we recall   some notations, definitions and basics results used throughout  the paper. A function $\varphi:\mathbb{R}^{n}\to  \mathbb{R} $ is said to be {\it convex}  if $\varphi(\lambda x + (1-\lambda)y)\leq \lambda \varphi(x) + (1-\lambda) \varphi(y)$, for all $x,y\in {\mathbb{R}^{n}}$ and  $\lambda \in [0,1]$, and  $\varphi$ is \emph{strictly convex}  when the  last inequality is strict for $x\neq y$. For a comprehensive study of convex function see \cite{Lemarechal}.   A continuously differentiable function $f: \mathbb{R}^n \to \mathbb{R}$ has an L-Lipschitz continuous gradient  $\nabla f$  on ${{\cal C}} \subset {\mathbb R}^n$, if    there  exists a Lipschitz  constant $L>0$ such that  $\| \nabla f(x)- \nabla f(y) \| \leq L\|x-y \|$ for all ~ $x,y \in {{\cal C}}$. Thus, by using the  fundamental theorem of calculus, we obtain the following result whose proof can be found in \cite[Proposition A.24]{Bertsekas1999}, see also\cite[Lemma~2.4.2]{DennisSchnabel1996}.
\begin{proposition}\label{le:cc}
The function $f:\mathbb{R}^{n}\to  \mathbb{R} $ is convex if, and only if,   $f(y)\geq f(x) + \langle \nabla f,(x) y-x \rangle$, for all $x,y\in \mathbb{R} ^{n} $.
\end{proposition}
\begin{proposition} \label{pr:DescentLemma}
	Let $f: \mathbb{R}^n \to \mathbb{R}$ be a differentiable  with gradient  $L$-Lipschitz continuous on ${{\cal C}} \subset {\mathbb R}^n$, $x \in {{\cal C}}$, $v\in \mathbb{R}^n$  and $\lambda \in [0,1]$. If  $x+\lambda v \in {{\cal C}}$, then 
$
f(x+\lambda v) \leq f(x) + \nabla f(x)^{\T}v \lambda + \frac{L}{2}\|v\|^2\lambda^{2}.
$
\end{proposition}


We end this section stating two  results for sequences of real numbers, which will be useful  for our study on iteration complexity bounds for the conditional gradient method.  Their proofs can be found in \cite[Lemma 6, Ch. 2, p. 48]{polyak1987} and \cite[Lemma 13.13, Ch. 13, p. 387]{Beck2017}, respectively.
\begin{lemma}\label{lemma taxa2}
Let $\{a_{k}\}_{k\in {\mathbb N}} $	be a nonnegative sequence of real numbers, if $\Gamma a_{k}^2 \leq a_{k} - a_{k+1}$ for some $\Gamma >0$ and for any $k=1,...,\ell$, then 
$$
a_\ell \leq \frac{a_0}{1+\ell\Gamma a_0}< \frac{ 1}{\Gamma \ell}.
$$
\end{lemma}
\begin{lemma}\label{lemma taxa}
Let $p$ be a positive integer, and let $\{a_{k}\}_{k\in {\mathbb N}}$	and $\{b_{k}\}_{k\in {\mathbb N}}$ be  nonnegative sequences of real numbers satisfying 
$$
a_{k+1}\leq a_{k}-b_{k}\beta_{k}+\frac{A}{2}\beta_{k}^{2}, \qquad k=0, 1, 2, \ldots, 
$$
where $\beta_{k}=2/(k+2)$ and $A$ is a positive number. Suppose that $a_{k}\leq b_{k}$, for all $k$. Then 
\begin{item}
\item[(i)] 	$\displaystyle a_k \leq \frac{2A}{k}$, for all $k=1, 2, \ldots.$
\item[(ii)] $\displaystyle \min_{\ell\in\{\lfloor\frac{k}{2}\rfloor+2,\cdot\cdot\cdot,k\}} b_\ell \leq \frac{8A}{k-2}$, for all $k=3, 4, \ldots,$ where, $\lfloor k/2 \rfloor= \max \left\lbrace n\in \mathbb{N} :~ n\leq k/2\right\rbrace.$ 

\end{item}
\end{lemma}
\section{Star-Convex Functions} \label{sec:StarConvex}
We briefly recall the notion of \emph{star-convex functions}, as introduced in~\cite{NesterovPolyak2006}, which extends classical convexity while allowing certain non-convex structures. This property will be central to our analysis. We also present illustrative examples to clarify how star-convex functions differ from both convex and general non-convex functions.

\begin{definition} \label{def:star-conv}
Let ${{\cal C}} \subset {\mathbb R}^n$ be a convex set. A function $f:{\mathbb R}^n \to\mathbb{R}$  is said to be star-convex in ${{\cal C}}$ if its set of global minima  $X^*$  on the set  ${{\cal C}}$  is not empty and for any $x^*\in X^*$  we have 
\begin{equation} \label{eq:star-conv}
f(\lambda x^*+(1-\lambda)x)\leq \lambda f(x^*)+(1-\lambda)f(x), \qquad \forall x\in {{\cal C}}, \forall \lambda\in [0, 1].
\end{equation} 
\end{definition}  
Every convex function with global minimizer set non-empty  is a star-convex function, but in general,  star-convex  functions need not be convex. In the following we present  two examples of  star-convex functions that are not convex, which appeared in   \cite{NesterovPolyak2006}.

\begin{example}
The function $f(t) = |t|(1 - e^{-|t|})$ is star-convex,  but not convex. Indeed, $f$ is differentiable and satisfies $f(0) = 0$. However, its second derivative changes sign, indicating that the function is not convex on any interval containing the origin.
\end{example}

\begin{example} \label{ex:qhc}
Consider the function $f(s,t) = s^2t^2 + s^2 + t^2$. This function is star-convex with respect to the origin but not convex. Although each term is nonnegative and $f(0,0) = 0$, the Hessian matrix of $f$ is not positive semidefinite everywhere, which precludes convexity.
\end{example}

Next we show that  Example~\ref{ex:qhc} is a particular instance of the more general case of suitable positively homogeneous function, which are  star-convex.

\begin{example}\label{ex:homog_star}
Let \( f:\mathbb{R}^n \to \mathbb{R} \) be continuous and \emph{positively homogeneous} of degree \(r\geq 1\), i.e. 
\[
     f(\lambda x)=\lambda^{r}f(x), 
     \qquad\forall\,\lambda>0,\;x\in\mathbb{R}^{n}, 
\]
and nonnegative, i.e.,   \( f(x)\ge 0\),  for all \(x\in\mathbb{R}^{n}\).  Then, $f$  is star-convex. Indeed, for any \(x\in\mathbb{R}^{n}\) and \(\lambda\in[0,1]\),
the homogeneity of degree \(r\geq 1\) gives
\(
     f(\lambda x)=\lambda^{r}f(x)\leq \lambda f(x), 
\)
because \(0\le\lambda^{r-1}\le 1\) and \(f(x)\ge 0\).  
Hence, due to $f(0)=0$, we have \(f\bigl((1-\lambda)0+\lambda x\bigr)\le (1-\lambda)f(0)+\lambda
f(x)\), for any \(x\in\mathbb{R}^{n}\). Therefore,   \(f\) is star-convex.
\end{example}

Let us now  present some concrete instances of Example~\ref{ex:homog_star}.
\begin{example}
Let $p\in\mathbb{R}$ be  a fixed and consider the  function $f:{\mathbb R}^n \to\mathbb{R}$  defined by 
\[
     f_{p}(x):= \Bigl(\sum_{i=1}^n|x_i|^p\Bigr)^{1/p}.
\]
The set of global  minimum   of $f_p$ is  given by $X^*=\{(0,0)\}$.  It is well known  that  for $p\ge 1$ the map $f_p$ is convex; it coincides with the
        $\ell_p$ norm in~$\mathbb{R}^{2}$. Now, for \emph{all} real $p$ (positive, negative, or zero\footnote{%
        When $p=0$, $f_0(x_1,x_2,\cdots,x_n)=\left(\Pi_{i=1}^{n} |x_i|\right)^{1/n}$ is obtained by
        continuity.}) the function $f_p$ is \emph{star-convex}. Indeed,  we have 
        $
            f_p\bigl(\lambda x\bigr)= \lambda f_p(x),
        $
        for all  $\lambda\in[0,1]$ and  all $x\in\mathbb{R}^{n}$, 
        which implies that $f_p$ is  homogeneous  of degree~$1$ when $p\ne 0$, and by continuity for $p=0$. Therefore,  considering that   \( f(x)\ge 0\)  for all \(x\in\mathbb{R}^{n}\), it follows from Example~\ref{ex:homog_star} that it is star-convex. 
\end{example}

\begin{example}
Let $0<r<1$ be  a fixed and consider the  function $f:{\mathbb R}^n \to\mathbb{R}$  defined by 
\[
     f_{r}(x):= \|x\|^r.
\]
The set of global  minimum   of $f_r$ is  given by $X^*=\{(0,0)\}$. We can verify that satisfies  
        $
            f_r\bigl(\lambda x\bigr)= \lambda^r f_r(x),
        $
        for all  $\lambda\in[0,1]$ and  all $x\in\mathbb{R}^{n}$, which implies that $f_r$ is  homogeneous  of degree~$r$. Therefore,  considering that   \( f_r(x)\ge 0\)  for all \(x\in\mathbb{R}^{n}\), it follows from Example~\ref{ex:homog_star} that it is star-convex. We can also verify that  $f_r$ is  concave.
\end{example}

\begin{example}\label{ex:star_convex_sets}
Let $\mathcal{C}_i\subset\mathbb{R}^{n}$ ($i=1,\dots,m$) be non-empty,
closed, and \emph{star-shaped} with respect to every point in their
common intersection  
\[
    X^{*}\;:=\;\bigcap_{i=1}^{m}\mathcal{C}_{i}\;\neq\;\varnothing.
\]
Let $ d_{\mathcal{C}_i}^{2}: \mathbb{R}^{n} \to \mathbb{R}$ be  the squared distance with respect to the set   ${\mathcal{C}_i}$ defined by 
\[
    d_{\mathcal{C}_i}^{2}(x):=\inf_{y\in\mathcal{C}_i}\|x-y\|^{2}.
\]
Choose non-negative weights $\omega_i$ with
$\sum_{i=1}^{m}\omega_i=1$ and define the function  $f: \mathbb{R}^{n} \to \mathbb{R}$  by 
\[
    f(x):=\sum_{i=1}^{m}\omega_i\,d_{\mathcal{C}_i}^{2}(x).
\]
The function \(f\) is star-convex, although in general it is \emph{not} convex. Indeed, fix \(x^{*}\in X^{*}\), any \(x\in\mathbb{R}^{n}\), and \(\lambda\in[0,1]\).
Set
$$
    z:=\lambda x^{*}+ (1-\lambda)x.
$$
Letting  \(y_i\in\operatorname*{arg\,min}_{y\in\mathcal{C}_i}\|x-y\|\) we have $d_{\mathcal{C}_i}(x)=\|x-y_i\|$.  Because each \(\mathcal{C}_i\) is star-shaped about \(x^{*}\) and $y_i\in\mathcal{C}_i$, we have 
$$
    z_i:=\lambda x^{*}+(1-\lambda)y_i\in\mathcal{C}_i.
$$
Thus,  we obtain
$
    d_{\mathcal{C}_i}(z)\leq \bigl\|z-z_i\bigr\|=\bigl\|(1-\lambda)(x-y_i)\bigr\|=  (1-\lambda) \bigl\|x-y_i\bigr\|=(1-\lambda)\,d_{\mathcal{C}_i}(x).
$
Squaring and using \((1-\lambda)^2\le(1-\lambda)\) gives $ d_{\mathcal{C}_i}^{2}(z)\leq (1-\lambda)\,d_{\mathcal{C}_i}^{2}(x)$, for all $ i=1,\dots,m$. Multiplying each bound by its weight~$\omega_i$, summing over~$i$, and using 
$f(x^{*}) = 0$ yields
\[
  f\bigl(\lambda x^{*}+(1-\lambda)x \bigr)\leq \lambda\,f(x^{*})+ (1-\lambda)\,f(x) ,
\]
which is precisely the defining inequality for star-convexity.
  Finally, if at
least one \(\mathcal{C}_i\) is non-convex, the sum
\(f=\sum_{i=1}^m\omega_i d_{\mathcal{C}_i}^{2}\) need not be convex.
\end{example}

\begin{proposition}\label{pr:pscf}
Let $f:\mathbb{R}^{n}\to\mathbb{R}$ be a differentiable  function.  Assume that the set of global minimum   of $f$ on the set ${{\cal C}}$, denoted by the set $X^*$, is non-empty, and let  $f^*$ be the minimum value of $f$ on the set ${{\cal C}}$. If $f$ is star-convex in ${{\cal C}}$, then   $f^*-f(x)\geq \nabla f(x)^{T}(x^*-x)$, for all $x\in {{\cal C}}$ and  $x^*\in X^*$.
\end{proposition}
\begin{proof}
Let $x\in {{\cal C}}$.  It follows from Definition~\ref{def:star-conv} that for a given  minimizer $x^*\in X^*$  of $f$  we have  $f(x^*)-f(x)\geq (f(x+\lambda (x^*-x))-f(x))/\lambda$. Then, taking the limit as $\lambda$ goes to $+0$ we  have $f(x^*)-f(x)\geq \nabla f(x)^{T}(x^*-x)$. Since $f^*=f(x^*)$  the desired inequality follows.
\end{proof}

\section{The optimization  problem} \label{sec;OptProbl}
We are interested in solving the following constrained optimization problem
\begin{equation}\label{pr:main}
\begin{array}{c}
\min_{x\in {{\cal C}}} f(x), 
\end{array}
\end{equation}
where  ${{\cal C}} \subset {\mathbb R}^n$ is a compact  and convex set,  $f:\mathbb{R}^n \to \mathbb{R} $ is a  continuously differentiable  convex  function and its   gradient    is {\it  $L$-Lipschitz continuous} on ${{\cal C}} \subset {\mathbb R}^n$, i.e.,    there  exists a Lipschitz  constant $L>0$ such that 
\begin{itemize}
\item[{\bf (A)}] $\| \nabla f(x)- \nabla f(y) \| \leq L\|x-y \|$ for all ~ $x,y \in {{\cal C}}$.
\end{itemize}
Since we are assuming that   ${\cal C} \subset {\mathbb R}^n$ is a compact set,   its {\it diameter}  is a finite number defined by
\begin{equation*} 
 \diam({\cal C}):= \max\left\{ \|x-y\|:~x, y\in {\cal C}\right\}. 
\end{equation*} 
Since   ${{\cal C}} \subset {\mathbb R}^n$ is a compact,  the study of  problem~\eqref{pr:main} is  bounded from below. Then, optimum value of the problem~\eqref{pr:main} satisfy 
$
+\infty < f^*:=\inf_{x\in {{\cal C}}} f(x)
$
and optimal set   ${{\cal C}}^*$  is non-empty.  The {\it first-order optimality condition } for problem~\eqref{pr:main} is stated as
\begin{equation} \label{eq:oc}
 \nabla f({\bar x})^{T}(x-{\bar x}) \geq 0 , \qquad  \quad \forall x\in {{\cal C}}.
\end{equation}
In general, the condition \eqref{eq:oc} is necessary but not sufficient for optimality.  A point ${\bar x}\in {{\cal C}}$ satisfying condition \eqref{eq:oc}   is called a {\it stationary point} to problem~\eqref{pr:main}. Consequently,  all  $x^*\in {{\cal C}}^*$ satisfies  \eqref{eq:oc}.

We conclude this section by introducing two auxiliary mappings that will be useful  for defining the Frank–Wolfe algorithm and analysing its convergence:
\begin{equation}\label{eq:LO}
   p(x)\;\in\;{\argmin}_{u\in\mathcal{C}}\; \nabla f(x)^T(u-x),  \qquad \quad   \omega(x)\;:=\;\nabla f(x)^{\top}\bigl(p(x)-x\bigr),
   \qquad \forall x\in\mathcal{C}.
\end{equation}

\begin{proposition}\label{prop:omega}
Assume {\bf(A)}.  Then, the scalar gap function $\omega:\mathcal{C}\to\mathbb{R}$ defined in
\eqref{eq:LO} satisfies $\omega(x)\le 0$ for every $x\in\mathcal{C}$. In addition, $\omega$ is continuous on $\mathcal{C}$.
\end{proposition}

\begin{proof}
For proving  the first statement   note that  because $u=x$ is feasible in~\eqref{eq:LO}, we  conclude that  $\omega(x)=\min_{u\in\mathcal{C}}\nabla f(x)^{\top}(u-x) \le \nabla f(x)^{\top}(x-x)=0$.   To prove the  second statement, fix $x\in\mathcal{C}$ and a sequence $\{x^{k}\}_{k\in {\mathbb N}}\subset\mathcal{C}$ with $\lim_{k\to +\infty}x^{k}=x$.  We will show  that $\lim_{k\to +\infty}\omega(x^{k})=\omega(x)$.    Since $ \omega(x^{k}) \leq \nabla f(x^{k})^{\top}\!\bigl(p(x)-x^{k}\bigr)$ and $\nabla f$ is continuous, we have $\limsup_{k\to\infty}\omega(x^{k})\le\omega(x)$. On the other hand, due to  $p(x^{k}) \in \mathcal{C}$ we have 
\[
  \omega(x)\leq \nabla f(x)^{\top}\!\bigl(p(x^{k})-x\bigr)=\omega(x^{k}) +\!\bigl(\nabla f(x)-\nabla f(x^{k})\bigr)^{\top}\!\bigl(p(x^{k})-x^{k}\bigr)+\nabla f(x)^{\top}(x^{k}-x).
\]
The last term tends to $0$  as $k$ goes to $+\infty$, and because $\|p(x^{k})-x^{k}\|\le\diam(\mathcal{C})$ and
$\nabla f$ is continuous, the second term also  tends to $0$, yielding
$
   \omega(x)\le\liminf_{k\to\infty}\omega(x^{k}).
$
Combining the upper and lower limits proves continuity.
\end{proof}

\section{Frank--Wolfe algorithm} \label{Sec:FW}
In this section, we present the classical Frank--Wolfe algorithm for solving problem~\eqref{pr:main} and analyze its convergence under four different stepsize strategies. The method is projection-free and relies on solving a linear subproblem at each iteration, making it well suited for large-scale problems. We assume that $f$ satisfies condition {\bf (A)}, but the algorithm does not require prior knowledge of the Lipschitz constant in three of the four stepsize rules considered: Armijo backtracking, an adaptive estimate via backtracking, and a diminishing stepsize rule. We also include the classical Lipschitz-based rule for comparison. We show that under star-convexity, the algorithm achieves an ${\cal O}(1/k)$ convergence rate for both function values and the duality gap, extending classical results to this broader setting.

To define the algorithm, we assume access to a linear optimization oracle (LO oracle) capable of minimizing linear functions over the feasible set ${\cal C}$. The algorithm is formally described below.\\

\hrule
\begin{algorithm} {\bf Frank--Wolfe (FW)  algorithm} \label{dAlg:CondGdfS}
\begin{footnotesize}
\begin{description}
\item[Step 0.] {\it Initialization:} Choose $x^0\in {\cal C}$ and initialize  $k\gets 0$.
\item [Step 1.] {\it Compute the search direction:} Compute an optimal solution $p(x^k)$ and the optimal value ${\omega}(x^k)$ as 
\begin{equation} \label{eq: opsv}
 p(x^k)\in \arg\min_{u\in {\cal C}} \nabla  f(x^k)^T  (u-x^k), \qquad 
 {\omega}(x^k):= \nabla  f(x^k)^T (p(x^k) - x^k). 
\end{equation} 
\item[ Step 2.] {\it Stopping criteria: } If ${\omega}(x^k)= 0$, then {\bf stop}.    
\item[ Step 3.] {\it Compute the stepsize and iterate:}  Define the search direction by $d(x^k):=p(x^k)-x^k$ and compute $\lambda_k \in (0, 1]$ (different strategies for the stepsize are considered)  and set 
\begin{equation}\label{eq:iteration}
x^{k+1}:=x^k+ \lambda_k d(x^k).
\end{equation}
\item[ Step 4.] {\it Beginning a new iteration:} Set $k\gets k+1$ and go to {\bf Step 1}.
\end{description}
\hrule
\end{footnotesize}
\end{algorithm}
\vspace{0.3cm}

The oracle direction \(p(x)-x\) is the classical Frank--Wolfe search direction, while the scalar gap \(\omega(x)\le 0\) measures how far the
point \(x\) is from stationarity (cf.\ Proposition~\ref{prop:omega}).  Hence the basic FW-algorithm  stops successfully when \(\omega(x^{k})=0\). From now on we assume that all iterates generated by {FW}-algorithm are \emph{non-stationary}, i.e.\ \(\omega(x^{k})<0\) for every
\(k=0,1,\dots\).  Consequently the method produces an infinite sequence \(\{x^{k}\}_{k\in {\mathbb N}}\subset\mathcal{C}\).  Because the update rule \(x^{k+1}=x^{k}+\lambda_{k}\bigl(p(x^{k})-x^{k}\bigr)\) uses \(\lambda_{k}\in(0,1]\) and \(\mathcal{C}\) is convex, induction shows that every iterate remains in~\(\mathcal{C}\). The convergence behaviour of FW-algorithm  depends critically on the choice of stepsize~\(\lambda_{k}\).  We study three well-established strategies described below.

\begin{armijo}[Armijo backtracking]\label{ls:armijo}
Take  $\beta \in (0, 1)$ and the initial trial stepsize ${\bar \lambda}_{0}=1$. For each $k$, compute the positive integer number $\ell_{k}$ such that
\begin{small}
\begin{equation} \label{eq:Adapt}
      \ell_{k}:=\min \left\{\ell\in {\mathbb N}  :~ f\big(x^{k}+\beta^{\ell} {\bar \lambda}_{k} (p(x^k)-x^k)\big)\leq  f(x^{k}) - {\zeta}\beta^{\ell}{\bar \lambda}_{k}|\omega(x^k)|\right\},
\end{equation}
\end{small}
and define the stepsize $\lambda_{k}:=\beta^{\ell_k} {\bar \lambda}_{k}$. Then, update the trial stepsize by ${\bar \lambda} _{k+1}:= \beta^{\ell_{k}-1}{\bar \lambda}_{k}$.
\end{armijo}
Note that,  by setting ${\bar \lambda}_k=1$ in  \eqref{eq:Adapt} yields the classic Armijo strategy. The idea behind choosing the stepsize as in \eqref{eq:Adapt},  is to reduce the number of function evaluations needed during the line search process, enhancing the efficiency of optimization, especially for large-scale problems.  Finally, it is worth noting that the rationale for selecting the stepsize as outlined in \eqref{eq:Adapt},   is to enhance flexibility in choosing trial stepsizes, providing greater adaptability compared to classic Armijo strategy.

We now introduce a practical strategy that does not require prior knowledge of the Lipschitz constant. This approach can be seen as a variant of the method proposed in \cite{BeckTeboulle2004}, which explicitly relies on the Lipschitz constant. In contrast, our method simultaneously determines the stepsize and an estimate of the Lipschitz constant using a backtracking procedure.

\begin{adaptative}\label{ls:adaptive}
This strategy does not use the value of the
Lipschitz constant $L$, even if it is known:
\begin{description} 
\item[Step 3.1:]  Consider $L_0>0$. Compute the stepsize $\lambda_{j} \in (0, 1]$ as follows 
 \begin{equation}\label{deq:fixed.steps}
\lambda_{j}=\mbox{min}\left\{1, \frac{|{\omega}(x^k)|}{{2^{j}L_k}  \|p^{k}-x^k\|^2}\right\}
                 :={\argmin}_{\lambda \in (0,1]}\left \{-|{\omega}(x^k)| \lambda+\frac{{2^{j}L_k} }{2} \|p^{k} -x^k\|^2 \lambda^2 \right \}.
\end{equation}

\item[Step 3.2:] If
\begin{equation}\label{deq:tests}
 f(x^k+ \lambda_{j}(p^{k}-x^k)) \leq f(x^k) -|{\omega}(x^k)| \lambda_{j}+\frac{{2^{j}L_k} }{2} \|p^{k} -x^k\|^2 \lambda_j^2, 
\end{equation}
then set $j_k=j$ and go to {\bf Step 3.3}. Otherwise, set $j=j+1$ and go to {\bf Step 3.1}.

\item[Step 3.3:] Set  $\lambda_k:= \lambda_{j_k}$ and  define the next approximation to the Lipschitz constant   $L_{k+1}$ as  
\begin{equation}\label{deq:iterations}
 L_{k+1}:=2^{j_k-1}L_k.
\end{equation}
\end{description}
\end{adaptative}
\begin{remark}\label{re:lips}
    If the Lipschitz constant $L$ is known, taking $L_0=L$ and $j_k=0$ for all $k=1,2,\cdots $, we recover the following:
    \begin{equation}\label{eq:adaptive}
   \lambda_{k}
   :=
   \min\!\Bigl\{1,\;
          \frac{|\omega(x^{k})|}{L\|p(x^{k})-x^{k}\|^{2}}\Bigr\}
   ={\argmin}_{\lambda\in(0,1]}
      \Bigl\{-\lambda\,|\omega(x^{k})|
             +\tfrac{L}{2}\lambda^{2}\|p(x^{k})-x^{k}\|^{2}\Bigr\}.
\end{equation}
\end{remark}
Finally, we present a classical diminishing stepsize rule that is fully explicit and does not depend on any problem-specific parameters.
\begin{diminishing}[Deterministic diminishing step]\label{ls:diminish} Take $\lambda_k=\beta_k$ where:
\[
   \beta_{k}:=\frac{2}{k+2},\qquad k=0,1,\dots
\]
A simple analytic rule that requires no problem data.
\end{diminishing}
Before concluding this section, it is worth noting that a direct application of Proposition~\ref{pr:DescentLemma} ensures that both the Armijo backtracking strategy and the Lipschitz-based adaptive stepsize rule are well defined, that is, $\ell_k$ and $ j_k $ a can be determined in a finite number of steps, respectively. Each of these stepsize strategies will be analyzed in the next section, where we establish their corresponding convergence properties.
\subsection{Iteration-complexity for Armijo's stepsize} \label{Sec:IntCopArmijo}
In this section we analyse the behaviour of the  sequence \(\{x^{k}\}_{k\in\mathbb{N}}\) produced by the Frank--Wolfe algorithm equipped with the Armijo backtracking rule. We begin by establishing  a lower bound on the accepted stepsizes, and then combine these ingredients with the star-convexity structure of \(f\) to derive  complexity estimates. For the complexity estimates we introduce two auxiliary constants
that depend only on problem data and Armijo parameters:
\begin{equation}\label{eq:gamma_def}
   \rho = \sup_{x\in\mathcal{C}}\|\nabla f(x)\|, \qquad \gamma= \min\! \Bigl\{\frac{1}{\rho\,\diam(\mathcal{C})}, \frac{2(1-\zeta)}{\beta L\,\diam(\mathcal{C})^{2}} \Bigr\},
\end{equation}
where $L$ is the Lipschitz constant from assumption~{\bf(A)},
$\zeta\in(0,1)$ is the Armijo parameter, and
$\beta\in(0,1)$ is the minimal backtracking reduction factor
(see line search.

\begin{lemma}\label{lem:lambda_lower_bound}
Let  $\{x^{k}\}_{k\in\mathbb{N}}$ be the sequence generated by the Frank--Wolfe algorithmwith the Armijo stepsize. Then the produced stepsizes satisfy $\lambda_{k}\;\ge\;\gamma\,|\omega(x^{k})|$, for all  $k=0,1,\dots.$
\end{lemma}
\begin{proof}
Fix $k$ and write $d^{k}:=p(x^{k})-x^{k}$. First we assume that  \(\lambda_{k}=1\). By definition of $\omega$, we have 
$
 |\omega(x^{k})|= -\nabla f(x^{k})^{\top}d^{k}\leq\|\nabla f(x^{k})\|\,\|d^{k}\|\leq \rho\,\diam(\mathcal{C}),
$
Therefore, by using  \eqref{eq:gamma_def} we conclude that $1\ge |\omega(x^{k})|/[\rho\,\diam(\mathcal{C})]\ge\gamma|\omega(x^{k})|$. Now,  we assume that \(0<\lambda_{k}<1\). From  the backtracking we have 
\[
   f(x^{k}+\beta \lambda_{k}d^{k})
   >
   f(x^{k})+\zeta \beta \lambda_{k}\,\omega(x^{k}).
\]
On the other hand,  Proposition~\ref{pr:DescentLemma} gives   \(f(x^{k}+\beta \lambda_{k} d^{k}) \leq f(x^{k})+\beta \lambda_{k}\,\omega(x^{k})+\tfrac{L}{2}(\beta \lambda_{k})^{2}\|d^{k}\|^{2}\), which combined wilt the last inequality  yields
$$
   \zeta \beta \lambda_{k}|\omega(x^{k})|  < \beta \lambda_{k} \omega(x^{k})  +\frac{L}{2}\|d^{k}\|^{2}(\beta \lambda_{k})^{2}.
$$
Because $\omega(x^{k})<0$ and $ \beta \lambda_{k}> 0$, the last inequality implies that 
$$
   (1-\zeta)|\omega(x^{k})| <\frac{L}{2}\|d^{k}\|^{2}  \beta \lambda_{k}\leq\frac{L}{2}\,\diam(\mathcal{C})^{2}  \beta \lambda_{k}.
$$
Therefore, it follows from the last inequality  and  by taking into account  the definition of $\gamma$ in   \eqref{eq:gamma_def} that  $ \lambda_{k}\geq  \frac{2(1-\zeta)}{\beta L\,\diam(\mathcal{C})^{2}} |\omega(x^{k})| \geq\gamma\,|\omega(x^{k})|.$ Both cases establish the claimed lower bound.
\end{proof}

We now establish the first iteration-complexity bound for the Frank--Wolfe algorithm using the Armijo stepsize strategy. Under star-convexity, we show that the method achieves a sublinear convergence rate of order ${\cal O}(1/k)$ for the function values.

\begin{theorem}\label{thm:armijo_rate}
Assume that $f$ is star-convexity on $\mathcal{C}$.  Let $\{x^{k}\}_{k\in\mathbb{N}}$ be the sequence generated by the FW algorithm with the Armijo stepsize.   Then, there holds
\begin{equation}\label{eq:rate_armijo}
    f(x^{k})-f^*\leq\frac{1}{\Gamma k} ,
    \qquad k=1,2,\dots .
\end{equation}
\end{theorem}
\begin{proof}
By the Armijo backtracking rule, Proposition~\ref{prop:omega}
(\(\omega(x^{k})\le 0\)), and the lower bound
\(\lambda_{k}\ge \gamma|\omega(x^{k})|\) from
Lemma~\ref{lem:lambda_lower_bound}, we obtain
\begin{equation}\label{eq:descent_scalar}
   f(x^{k+1})
   \;\le\;
   f(x^{k})-\zeta\,\lambda_{k}\,|\omega(x^{k})|
   \;\le\;
   f(x^{k})-\zeta\gamma\,\omega(x^{k})^{2},
   \qquad k=0,1,\dots .
\end{equation}
Because  $f$ is star-convexity, it follows from Proposition~\ref{pr:pscf} that for each $x^{k}$ and   $x^{\ast}$  a global minimizer, we have 
$$
f(x^{*})-f(x^{k})\geq  \nabla f(x^{k})^{\top}(x^{*}-x^{k})\geq  \nabla f(x^{k})^{\top}(p(x^{k})-x^{k}) =-|\omega(x^{k})|,
$$
Therefore, setting $f^*=f(x^{*})$, we obtain that  
$
 0\leq a_{k}:=f(x^{k})-f(x^{\ast})\leq |\omega(x^{k})|.
$
Combining with \eqref{eq:descent_scalar}  gives $a_{k+1}\leq  a_{k}-\zeta\gamma a_{k}^{2}$, for all $k=0,1,\dots .$
Applying item (i) of Lemma~\ref{lemma taxa2} with \(\Gamma:=\zeta\gamma\) yields \eqref{eq:rate_armijo}. 
\end{proof}

\subsection{Iteration-complexity for Lipschitz-based and diminishing stepsizes} \label{Sec:ConvAnal}
In this section, we present iteration-complexity bounds for the sequence $(x^k)_{k\in{\mathbb N}}$ generated by the Frank--Wolfe algorithm with a Lipschitz-based adaptive and diminishing stepsizes, assuming that the objective function $f$ is star-convexity. Before stating the complexity result, we introduce a preliminary result. Since its proof is similar to those established in \cite{BeckTeboulle2004, pedregosa2020linearly}, we omit it. For notational simplicity, let $L$  the Lipschitz constant and $L_0>0$, we define the constant
\begin{equation*}
 \alpha := 2(L + L_0)\, \mathrm{diam}({\cal C})^2 > 0.
\end{equation*}

\begin{proposition} \label{dpr:mainDf}
Let $\{x^{k}\}_{k\in\mathbb{N}}$ be the sequence generated by the Frank--Wolfe algorithm with a Lipschitz-based adaptive stepsize. Then, for every $j\in {\mathbb N}$ such that $2^{j}L_k \geq L$, the inequality \eqref{deq:tests} holds. Consequently, the integer $j_k$ is well defined. Moreover, $j_k$ is the smallest non-negative integer satisfying the following two conditions:
$
2^{j_k}L_k \geq 2L_0,
$
and for any $\beta_k\in(0,1]$:
\begin{equation}\label{eq:TestLineSeachC2}
f(x^k + \lambda_k(p^{k} - x^k)) \leq f(x^k) - |{\omega}_{k}|\beta_k + \frac{2^{j_k}L_k}{2} \|p^k - x^k\|^2 \beta_k^2.
\end{equation}
In addition, the sequence $\{x^k\}_{k\in\mathbb{N}}$ generated by the Frank--Wolfe algorithm with Lipschitz-based adaptive stepsize is well defined. Furthermore, the following inequality holds for all $k \geq 0$:
\[
f(x^{k+1}) \leq f(x^k) - \frac{1}{2} |{\omega}_{k}| \lambda_k.
\]
In particular, the sequence $\{L_k\}_{k\in\mathbb{N}}$ satisfies the bounds
$L_0 \leq L_k \leq L + L_0,$ for all $ k \in \mathbb{N}$, and the stepsize sequence $\{\lambda_k\}_{k\in\mathbb{N}}$ satisfies
\begin{equation} \label{eq:TestLineSeachC2b}
\lambda_k \geq \min\left\{1, \frac{|{\omega}_{k}|}{\alpha} \right\}, \quad \forall\, k \in \mathbb{N}.
\end{equation}
\end{proposition}
When the Lipschitz constant is known, we can recall Remark \ref{re:lips} and the following result is consequence of Proposition \ref{dpr:mainDf}.
\begin{lemma}\label{l:icb}
Let  $\{x^{k}\}_{k\in\mathbb{N}}$ be the sequence generated by the FW algorithm using either the Lipschitz-based or the diminishing stepsize rule. Then, for all $k \in \mathbb{N}$,
\begin{equation}\label{brasil}
f(x^k+ \lambda_k(p(x^k)-x^k)) \leq  f(x^k)-|\omega(x^k)|\, \beta_{k} + \frac{L}{2}\|p(x^k) -x^k\|^2\, \beta_k^2, 
\end{equation}
where $\beta_k \in (0,1]$, particularly $\beta_k=2/(k+2)$, which is common in the literature.
\end{lemma}
We now present a complexity bound for the FW algorithm applied to star-convex functions over compact convex sets, using a Lipschitz-based adaptive stepsize rule. The result guarantees sublinear rates for both the function value suboptimality and the optimality measure.
\begin{theorem} \label{th:fcr}
Assume that \( f:\mathbb{R}^n \to \mathbb{R} \) is star-convex on the compact convex set \(\mathcal{C}\), and that  $\{x^{k}\}_{k\in\mathbb{N}}$ is a  sequence generated by the  FW algorithm using either the Lipschitz-based or the diminishing stepsizes. Then,
\begin{item}
\item[(i)] 	$f(x^k)-f^*  \leq \displaystyle   \frac{4(L+L_0)\diam({\cal C})^2}{k}$, for all $k=1, 2, \ldots.$
\item[(ii)] $\displaystyle \min_{\ell \in \left\{ \lfloor \frac{k}{2}\rfloor+2,...,k \right\} } |{\omega}_{\ell}| \leq \frac{16(L+L_0)\diam({\cal C})^2}{k-2}$, for all $k=3, 4, \ldots,$ where $\lfloor k/2 \rfloor= \max_{n\in \mathbb{N}} \left\{ n\leq k/2\right\}.$ 
\end{item}
 \end{theorem}
\begin{proof}
 It follows from  \eqref {eq:TestLineSeachC2} in  Proposition~\ref{dpr:mainDf}   that 
\begin{equation}  \label{eq;ascfa}
 f(x^k+ \lambda_{k}(p^{k}-x^k)) \leq f(x^k) -|{\omega}(x^k)|\lambda_{k}+\frac{{2^{j_k}L_k} }{2} \|p^{k} -x^k\|^2\lambda_{k}^2. 
\end{equation}
On the other hand, by using \eqref{deq:fixed.steps} we conclude that  
\begin{equation*}
\lambda_{k}={\argmin}_{\lambda \in (0,1]}\left \{-|{\omega}(x^k)| \lambda+\frac{{2^{j_k}L_k} }{2} \|p^{k} -x^k\|^2 \lambda^2 \right \}.
\end{equation*}
Hence, takinging  $\beta_k\in (0, 1]$,  it follows from  \eqref{eq;ascfa} and the last inequality that 
  \begin{equation*} 
  f(x^k+ \lambda_{k}(p^{k}-x^k)) \leq f(x^k) -|{\omega_k}| \beta_k + \frac{{2^{j_k}L_k}}{2}\|p^{k}-x^k\|^2\beta_k^{2}.
 \end{equation*}
Since $ \|p(x^k) -x^k\| \leq \diam({\cal C})$, the last inequality together with  \eqref{deq:iterations} and  inequality $ L_k \leq L + L_0,$ in Proposition~\eqref{dpr:mainDf} yield 
  \begin{equation}  \label{eq:fiicb}
  f(x^{k+1})-f^* \leq f(x^k)-f^* -|{\omega_k}| \beta_k + (L+L_0)\diam({\cal C})^2\beta_k^{2}.
 \end{equation}
 Taking into account that $f$ is   a  star-convex function in ${{\cal C}}$, it follows from Proposition~\ref{pr:pscf} that for  $x^*$  a global minimizer  we have  
 $$
 f^*-f(x^k) \geq \nabla f(x^k)^{\T}(x^* - x^k).
 $$
 Thus,  it follows from \eqref{eq: opsv}  that   $0\geq f^*-f(x^k)  \geq {\omega}(x_{k})$, which implies that  $0\leq f(x^k) - f^*\leq  |{\omega_k}|$.
Thus,  setting 
$$
a_{k}:= f(x^k)-f^*\leq b_{k}:=|{\omega_k}|, \qquad \alpha:=2(L+L_0)\diam({\cal C})^2, 
$$
we obtain from  using \eqref{eq:fiicb}   that 
$
a_{k+1} \leq a_k - b_k \beta_k +  \alpha \beta_k^2.
$
  applying   Lemma \ref{lemma taxa} w gives the desired inequalities. 
\end{proof}
According to Theorem~\ref{th:fcr}, functions with the star-convexity property allow the Frank-Wolfe algorithm to efficiently minimize them even when their landscape is not convex.
\section{Conclusions} \label{sec:conclusions}

We analyzed the iteration-complexity properties of the Frank--Wolfe algorithm for optimization problems with star-convexity objective functions. Under this generalized convexity assumption, we proved that the algorithm achieves an ${\cal O}(1/k)$ convergence rate for both the objective function values and the duality gap—matching the classical bounds known for convex objectives and confirming the method’s robustness in broader settings. A central aspect of our analysis is the choice of stepsize: we examined both a predefined diminishing rule and an adaptive strategy based on Lipschitz estimates. The adaptive rule, in particular, enables explicit stepsize computation without backtracking, while preserving worst-case guarantees. These findings reinforce the practical relevance of Frank--Wolfe methods in large-scale or structured problems, where projection steps are computationally costly. Overall, our results advance the theoretical understanding of projection-free algorithms under relaxed convexity assumptions and point to new opportunities for their application in nonstandard optimization scenarios.

\section*{Funding}
The first and third authors were supported by the Australian Research Council (ARC), Solving hard Chebyshev approximation problems through nonsmooth analysis (Discovery Project DP180100602). The second author was supported in part by CNPq - Brazil Grants 304666/2021-1. This work was done, in part, while the second author visited the first and third authors at Deakin University in November 2022. The second author thanks the host institution for funding the visit and for the pleasant scientific atmosphere it provided during his visit.

\bibliographystyle{habbrv}
\bibliography{WeakStarConvex}

\begin{thebibliography}{10}

\bibitem{Beck2017}
A.~Beck.
\newblock {\em First-order methods in optimization}, volume~25 of {\em MOS-SIAM Series on Optimization}.
\newblock Society for Industrial and Applied Mathematics (SIAM), Philadelphia, PA; Mathematical Optimization Society, Philadelphia, PA, 2017.

\bibitem{Beck2015}
A.~Beck, E.~Pauwels, and S.~Sabach.
\newblock The cyclic block conditional gradient method for convex optimization problems.
\newblock {\em SIAM J. Optim.}, 25(4):2024--2049, 2015.

\bibitem{BeckTeboulle2004}
A.~Beck and M.~Teboulle.
\newblock A conditional gradient method with linear rate of convergence for solving convex linear systems.
\newblock {\em Math. Methods Oper. Res.}, 59(2):235--247, 2004.

\bibitem{BeckTeboulle2009}
A.~Beck and M.~Teboulle.
\newblock A fast iterative shrinkage-thresholding algorithm for linear inverse problems.
\newblock {\em SIAM J. Imaging Sci.}, 2(1):183--202, 2009.

\bibitem{Bertsekas1999}
D.~P. Bertsekas.
\newblock {\em Nonlinear programming}.
\newblock Athena Scientific Optimization and Computation Series. Athena Scientific, Belmont, MA, second edition, 1999.

\bibitem{Bouhamidietall2018}
A.~Bouhamidi, M.~Bellalij, R.~Enkhbat, K.~Jbilou, and M.~Raydan.
\newblock Conditional gradient method for double-convex fractional programming matrix problems.
\newblock {\em J. Optim. Theory Appl.}, 176(1):163--177, 2018.

\bibitem{BoydRecht2017}
N.~Boyd, G.~Schiebinger, and B.~Recht.
\newblock The alternating descent conditional gradient method for sparse inverse problems.
\newblock {\em SIAM J. Optim.}, 27(2):616--639, 2017.

\bibitem{DennisSchnabel1996}
J.~E. Dennis, Jr. and R.~B. Schnabel.
\newblock {\em Numerical methods for unconstrained optimization and nonlinear equations}, volume~16 of {\em Classics in Applied Mathematics}.
\newblock Society for Industrial and Applied Mathematics (SIAM), Philadelphia, PA, 1996.
\newblock Corrected reprint of the 1983 original.

\bibitem{FrankWolfe1956}
M.~Frank and P.~Wolfe.
\newblock An algorithm for quadratic programming.
\newblock {\em Naval Res. Logist. Quart.}, 3:95--110, 1956.

\bibitem{FreundMazumder2017}
R.~M. Freund, P.~Grigas, and R.~Mazumder.
\newblock An extended {F}rank-{W}olfe method with ``in-face'' directions, and its application to low-rank matrix completion.
\newblock {\em SIAM J. Optim.}, 27(1):319--346, 2017.

\bibitem{Ghadimi2019}
S.~Ghadimi.
\newblock Conditional gradient type methods for composite nonlinear and stochastic optimization.
\newblock {\em Math. Program.}, 173(1-2, Ser. A):431--464, 2019.

\bibitem{HarchaouiNemirovski2015}
Z.~Harchaoui, A.~Juditsky, and A.~Nemirovski.
\newblock Conditional gradient algorithms for norm-regularized smooth convex optimization.
\newblock {\em Math. Program.}, 152(1-2, Ser. A):75--112, 2015.

\bibitem{Lemarechal}
J.-B. Hiriart-Urruty and C.~Lemar\'{e}chal.
\newblock {\em Convex analysis and minimization algorithms. {I}}, volume 305 of {\em Grundlehren der mathematischen Wissenschaften [Fundamental Principles of Mathematical Sciences]}.
\newblock Springer-Verlag, Berlin, 1993.
\newblock Fundamentals.

\bibitem{Jaggi2013}
M.~Jaggi.
\newblock Revisiting frank-wolfe: Projection-free sparse convex optimization.
\newblock {\em Proceedings of the 30th International Conference on International Conference on Machine Learning - Volume 28}, ICML'13:I--427--I--435, 2013.

\bibitem{Konnov2018}
I.~V. Konnov.
\newblock Simplified versions of the conditional gradient method.
\newblock {\em Optimization}, 67(12):2275--2290, 2018.

\bibitem{Lacoste2016}
S.~Lacoste-Julien.
\newblock Convergence rate of frank-wolfe for non-convex objectives.
\newblock {\em arXiv preprint arXiv:1607.00345}, 2016.

\bibitem{LanZhou2016}
G.~Lan and Y.~Zhou.
\newblock Conditional gradient sliding for convex optimization.
\newblock {\em SIAM J. Optim.}, 26(2):1379--1409, 2016.

\bibitem{LevitinPolyak1966}
E.~Levitin and B.~Polyak.
\newblock Constrained minimization methods.
\newblock {\em USSR Comput. Math. Math. Phys.}, 6(5):1--50, 1966.

\bibitem{LussTeboulle2013}
R.~Luss and M.~Teboulle.
\newblock Conditional gradient algorithms for rank-one matrix approximations with a sparsity constraint.
\newblock {\em SIAM Rev.}, 55(1):65--98, 2013.

\bibitem{NesterovPolyak2006}
Y.~Nesterov and B.~T. Polyak.
\newblock Cubic regularization of {N}ewton method and its global performance.
\newblock {\em Math. Program.}, 108(1, Ser. A):177--205, 2006.

\bibitem{pedregosa2020linearly}
F.~Pedregosa, G.~Negiar, A.~Askari, and M.~Jaggi.
\newblock Linearly convergent frank--wolfe with backtracking line-search.
\newblock In {\em Proc. Int. Conf. Artif. Intell. Statist. (AISTATS)}, pages 1--10. PMLR, 2020.

\bibitem{PedregosaJaggi2020}
F.~Pedregosa, G.~Negiar, A.~Askari, and M.~Jaggi.
\newblock Linearly convergent frank-wolfe with backtracking line-search.
\newblock {\em International Conference on Artificial Intelligence and Statistics}, pages 1--10, 2020.

\bibitem{polyak1987}
B.~T. Polyak.
\newblock {\em Introduction to Optimization}.
\newblock Translations Series in Mathematics and Engineering. Optimization Software, New York, 1987.

\end{thebibliography}

\end{document}